\documentclass[12pt]{article}
\usepackage{amsfonts, amsmath, amsthm, amssymb, mathrsfs}
\usepackage{fullpage}
\usepackage{MnSymbol}

\newcommand{\ol}[1]{\overline{\vphantom{1}#1}}
\newcommand{\I}{\infty}

\theoremstyle{Theorem}
\newtheorem{thm}{Theorem}[section]
\newtheorem{cor}[thm]{Corollary}
\newtheorem{lem}[thm]{Lemma}

\theoremstyle{definition}

\newtheorem{dff}[thm]{Definition}

\title{Explicit Hilbert-Kunz functions of $2 \times 2$ determinantal rings}
\author{Marcus Robinson and Irena Swanson}
\date{}

\begin{document}

\maketitle

\begin{abstract}
\noindent
Let $k[X] = k[x_{i,j}: i = 1, \ldots, m; j = 1, \ldots, n]$
be the polynomial ring in $m n$ variables $x_{i,j}$
over a field $k$ of arbitrary characteristic.
Denote by $I_2(X)$ the ideal generated by the $2 \times 2$ minors of 
the generic $m \times n$ matrix $[x_{i,j}]$.
We give a closed formulation for the dimensions of the $k$-vector space
$k[X]/(I_2(X) + (x_{1,1}^q, \ldots, x_{m,n}^q))$
as $q$ varies over all positive integers,
i.e.,
we give a closed form for the generalized Hilbert-Kunz function
of the determinantal ring $k[X]/I_{2}[X]$.
We also give a closed formulation of dimensions
of related quotients of $k[X]/I_{2}[X]$.
In the process we establish a formula
for the numbers of some compositions (ordered partitions of integers),
and we give a proof of a new binomial identity.
\end{abstract}

\section{Introduction}

\def\length{\hbox{length}\,}

Throughout let $m, n, q$ be non-negative integers
and $k$, $k[X]$, and $I_2(X)$ as in the abstract.
The generalized Hilbert-Kunz function
of $R = k[X]/I_{2}[X]$
is the function $\mbox{HK}_{R,X} : \mathbb{N} \to \mathbb{N}$ given by
$$
\mbox{HK}_{R,X}(q)
= \length \displaystyle\left( \frac {R} {x_{1,1}^q, \ldots, x_{m,n}^q) }\right)
= \length \displaystyle\left( \frac {k[X]} {I_{2}(X) + (x_{1,1}^q, \ldots, x_{m,n}^q) }\right).
$$
The standard Hilbert-Kunz function is only defined
when $k$ has positive prime characteristic~$p$
and when $q$ varies over powers of $p$,
whereas the generalized Hilbert-Kunz function
is defined for arbitrary field $k$, regardless of the characteristic.
While the Hilbert-Kunz function is not necessarily a polynomial function,
it has a well-defined normalized leading coefficient.
The normalized leading coefficient of the generalized Hilbert-Kunz function
has been studied by several authors,
including
Conca \cite{Con96},
Eto \cite{Eto02},
Eto and Yoshida \cite{EY03},
while
Miller and Swanson \cite{MS}
studied the whole generalized Hilbert-Kunz function.
Miller and Swanson 
gave a recursive formulation for $\mbox{HK}_{R,X}$
and proved that it is a polynomial function.
They gave closed formulations in the case $m \leq 2$.
This paper is an extension of~\cite{MS}.

The main result of this paper, Theorem~\ref{thmNinfinf},
is the closed formulation of
$\mbox{HK}_{R,X}$
for arbitrary positive integers $m, n$.
We also give,
in Theorem~\ref{thm:nq},
an explicit length of
$$
\frac{k[X]} {I_{2}(X) + (x_{i,j}^{q} : i ,j)
+ \sum_{j=1}^{n} (x_{1,j}, \ldots, x_{m,j})^q }.
$$
In Lemma~\ref{lemWU} and Corollary~\ref{corcompid}
we give some explicit formulas for the number of
tuples of specific length
of non-negative integers that sum up to at most a fixed number
and whose first few entries are at most another fixed number.
(In other words,
we give formulas for the numbers of some specific compositions of integers.)

\section{Set-Up}

Our proofs are based on the following result from \cite{MS}:

\begin{thm}\label{thmvsbasis}
\cite[Theorem 2.4]{MS}
A k-vector space basis for $k[X]/(I_{2} + (x_{1,1}^q, \ldots, x_{m,n}^q))$
consists of monomials $\prod_{i,j} x_{i,j}^{p_{i,j}}$
with the following properties:
\begin{enumerate}
\item
Whenever $p_{i,j} >0$ and $i' <i, j <j'$,
then $p_{i',j'} = 0$.
(Monomials satisfying this property will be called {\bf staircase monomials}.
The name comes from the southwest-northeast staircase-like shape
of the non-zero $p_{i,j}$ in the $m\times n$ matrix of all $p_{i,j}$.)
\item
Either for all $i = 1, \ldots, m,$ $\sum_{j} p_{i,j} < q$,
or for all $j =1, \ldots, n,$ $\sum_{i} p_{i,j} < q$.
\end{enumerate}
\qed
\end{thm}

Thus to compute the Hilbert-Kunz function,
we need to be able to count such monomials.
The recursive formulations for this function in \cite{MS},
as well as the explicit formulations below,
require counting related sets of monomials:

\begin{dff}\label{defNq}
\cite[Section 3]{MS}
Let $r_{1}, \ldots, r_{m}, c_{1}, \ldots,c_{n} \in \mathbb{N} \cup \{\infty\}$.
(In general we think of the $r_i$ as the row sums
and the $c_j$ as the column sums.)
Define $N_{q}(m,n;r_{1}, \ldots, r_{m}; c_{1}, \ldots, c_{n})$
to be the number of monomials $\prod_{i,j}x_{i,j}^{p_{i,j}}$
with the following properties:
\begin{enumerate}
\item
$\prod_{i,j}x_{i,j}^{p_{i,j}}$ is a staircase monomial,
i.e.,
whenever $p_{i,j} >0$ and $i' <i, j <j'$,
then $p_{i',j'} = 0$.
\item
For all $i \in \{1,\ldots, m\},$ $\sum_{j}p_{i,j} \leq r_{i}$,
and for all $j \in \{1,\ldots,n\},$ $\sum_{i}p_{i,j} \leq c_{j}$.
\item
Either for all $i \in \{1,\ldots, m\},$ $\sum_{j}p_{i,j} < q$,
or for all $j \in \{1,\ldots,n\},$ $\sum_{i}p_{i,j} < q$.
\end{enumerate}

For ease of notation let $\overline{c}$ denote $\{c, c, \ldots,c \}$.
For example,
$N_{q}(m,n; \I,\ldots, \I;\I, \ldots, \I)= N_{q}(m,n; \overline{\I}, \overline{\I})$.
By convention, $N_{q}(0,n;;c_{1}, \ldots, c_{n}) =1$.
\end{dff}

It was proved in \cite[Section 3]{MS} that
$N_q(m,n;r_1, \ldots, r_m; c_1, \ldots, c_n)$ equals
$$
\length \left(
{k[X] \over
I_2(X) + (x_{i,j}^q: i, j) + \sum_{i=1}^m (x_{i,1}, \ldots, x_{i,n})^{r_i + 1}
+ \sum_{j=1}^n (x_{1,j}, \ldots, x_{m,j})^{c_j + 1}
}\right),
$$
where for an ideal $I$,
we set $I^{\infty}$ to be the $0$ ideal.
Thus in particular
$N_{q}(m,n; \overline{\I};\overline{\I})
= \mbox{HK}_{K[X]/I_{2}(X), X} (q)$.

Our main result relies on the 
count of the following monomials as well:

\begin{dff}\label{defMq}
Let $r_{1}, \ldots, r_{m}, c_{1}, \ldots,c_{n} \in \mathbb{N} \cup \{\infty\}$.
Define $M_{q}(m,n; r_{1}, \ldots, r_{m}; c_{1}, \ldots, c_{n})$
to be the number of monomials $\prod_{i,j}x_{i,j}^{p_{i,j}}$ such that
\begin{enumerate}
\item
$\prod_{i,j}x_{i,j}^{p_{i,j}}$ is a staircase monomial,
i.e.,
whenever $p_{i,j} >0$ and $i' <i, j <j'$,
then $p_{i',j'} = 0$.
\item
For all $i \in \{1,\ldots, m\},$ $\sum_{j}p_{i,j} \leq \min\{r_i,q-1\}$.
\item
There exists $j \in \{1,\ldots,n\}$ such that $\sum_{i}p_{i,j} > c_{j}$.
\end{enumerate}
\end{dff}

The following lemma says that $m n$ exponents $p_{i,j}$ of a staircase monomial
can be identified by $m+n$ or even $m+n-1$ numbers:

\begin{lem}\label{lemstairmargs}
Suppose that $r_1, \ldots, r_m, c_1, \ldots, c_n$ are non-negative integers
and that $\sum_i r_i = \sum_j c_j$.
Then there exists a unique 
staircase monomial $\prod_{i,j} x_{i,j}^{p_{i,j}}$
such that for all $i = 1, \ldots, m$,
$r_i = \sum_j p_{i,j}$
and that
for all $j = 1, \ldots, n$,
$c_j = \sum_i p_{i,j}$.
\end{lem}

\begin{proof}
If $m = 1$,
the clearly $p_{1,j} = c_j$,
which is uniquely determined.
If $n = 1$,
necessarily $p_{i,1} = r_i$.

In general,
for arbitrary $m$ and $n$,
knowing $c_1$ and $r_m$ is enough information to uniquely determine $p_{m,1}$:
If $p_{m,1} < \min\{c_1, r_m\}$,
then the $m$th row has a non-zero number beyond the first entry
and the first column has a non-zero number in the first $m-1$ rows,
which then makes the corresponding monomial non-staircase
and is not allowed.
So necessarily $p_{m,1} = \min\{c_1, r_m\}$.
If $p_{m,1} = c_1$,
then no more non-zero exponents appear in the first column,
and it remains to fill in the remaining $m \times (n-1)$ matrix of $p_{i,j}$
with the remaining numbers
$r_1, \ldots, r_{m-1}, r_m - c_1, c_2, \ldots, c_n$.
If instead $p_{m,1} = r_m$,
then no more non-zero exponents appear in the last row,
and it remains to fill in the remaining $(m-1) \times n$ matrix of $p_{i,j}$
with the remaining numbers
$r_1, \ldots, r_{m-1}, c_1 - r_m, c_2, \ldots, c_n$.
\end{proof}

\begin{lem}\label{lemWU}
Let $a, b, w, z$ be integers with $a \le b$.
The number of $b$-tuples of non-negative integers
that sum up to at most $w$
and for which the first $a$ entries are strictly smaller than~$z$
equals
$$
\sum_{i=0}^a (-1)^i \binom a i \binom{w-iz + b} b.
$$
\end{lem}

\proof
Set
$U_{w,l} = \{(v_1, \ldots, v_b) \in \mathbb N^b:
\sum_{i=1}^b v_i \le w, \mbox{ and }
l \le |\{i \le a: v_i \ge z\}|\}$.
The desired number is $|U_{w,0} \setminus U_{w,1}|$.

Certainly $|U_{w,0}| = \binom{w+b}b$,
which is the number of ways of positioning up to $w$ stones into $b$ boxes
(possibly more than one stone per box).

Let $(v'_1, \ldots, v'_b) \in U_{w-lz,0}$
and let $L$ be an $l$-element subset of $\{1, \ldots, a\}$.
We construct $(v_1, \ldots, v_b) \in \mathbb{N}^b$
such that $v_i = v'_i + z$ if $i \in L$
and otherwise $v_i = v'_i$.
Note that $(v_1, \ldots, v_b) \in U_{w,l}$.
Every element $(v_1, \ldots, v_b)$ of $U_{w,l}$ is obtained in this way,
but note that for $k \ge l$,
each element $(v_1, \ldots, v_b)$ of $U_{w,k} \setminus U_{w,k+1}$
is obtained in this way in exactly $\binom k l$ ways.
This proves that
$$
\sum_{k=l}^a \binom k l |U_{w,k} \setminus U_{w,k+1}|
= \binom a l |U_{w-lz,0}|
= \binom a l \binom {w-lz + b} b.
$$
This gives a system of $a+1$ linear equations
whose matrix form is
$$
\begin{bmatrix}
\binom 0 0 & \binom 1 0 & \binom 2 0 & \cdots & \binom a 0 \cr
\binom 0 1 & \binom 1 1 & \binom 2 1 & \cdots & \binom a 1 \cr
\binom 0 2 & \binom 1 2 & \binom 2 2 & \cdots & \binom a 2 \cr
\vdots & \vdots & \vdots & \ddots & \vdots \cr
\binom 0 a & \binom 1 a & \binom 2 a & \cdots & \binom a a \cr
\end{bmatrix}
\begin{bmatrix}
\vphantom{\binom 0 0} |U_{w,0} \setminus U_{w,1}| \cr
\vphantom{\binom 0 0} |U_{w,1} \setminus U_{w,2}| \cr
\vphantom{\binom 0 0} |U_{w,2} \setminus U_{w,3}| \cr
\vdots \cr
\vphantom{\binom 0 0} |U_{w,a} \setminus U_{w,a+1}| \cr
\end{bmatrix}
=
\begin{bmatrix}
\binom a 0 \binom {w-0z + b} b \cr
\binom a 1 \binom {w-1z + b} b \cr
\binom a 2 \binom {w-2z + b} b \cr
\vdots \cr
\binom a a \binom {w-az + b} b \cr
\end{bmatrix}.
$$
The $(a+1) \times (a+1)$ matrix above is upper triangular
with entries $1$ on the diagonal.
It is easy to see that the inverse of this matrix $[\binom{j-1}{i-1}]_{i,j}$
is
$[(-1)^{i+j}\binom{j-1}{i-1}]_{i,j}$.
Thus by Cramer's rule,
$$
|U_{w,0} \setminus U_{w,1}|
= \sum_{i=1}^{a+1} (-1)^{i-1} \binom a {i-1} \binom{w-(i-1)z + b} b
= \sum_{i=0}^a (-1)^i \binom a i \binom{w-iz + b} b.
\qed
$$

\section{Main theorems}

In this section we give explicit formulas for
$N_{q}(m,n;\overline{\I};\overline{q-1})$,
$M_{q}(m,n;\overline{q-1};\overline{q-1})$,
and
$N_{q}(m,n;\overline{\I};\overline{\I})$
for arbitrary positive integers $m, n$.

\begin{thm}\label{thm:nq}
For all $m, n$,
the length of $\displaystyle \frac{k[X]} {I_{2}(X) + (x_{i,j}^{q} : i ,j)
+ \sum_{j=1}^{n} (x_{1,j}, \ldots, x_{m,j})^q }$
equals
$$
N_{q}(m,n; \ol{\infty};\ol{q-1})
= \sum_{i=1}^{n} (-1)^{n-i} \binom {n} {i} \binom {iq +m-1} {m+n-1},
$$
and furthermore,
this number equals the number of $(m+n-1)$-tuples
of non-negative integers that sum up to at most $n(q-1)$
and for which the first $n$ entries are strictly smaller than $q$.
\end{thm}

\proof
Let $T_{m,n,q}$ be the set of all staircase monomials
$\prod_{i,j} x_{i,j}^{p_{i,j}}$
such that for all $j = 1, \ldots, n$,
$\sum_i p_{i,j} < q$.
By \cite[Section 3]{MS},
$$
|T_{m,n,q}| = N_{q}(m,n;\overline{\I};\overline{q-1})
= \length \left(\displaystyle \frac{k[X]} {I_{2}(X) + (x_{i,j}^{q} : i ,j)
+ \sum_{j=1}^{n} (x_{1,j}, \ldots, x_{m,j})^q }\right).
$$
Let $W$ be the set of $(m+n-1)$-tuples
of non-negative integers that sum up to at most $n(q-1)$
and for which the first $n$ entries are strictly smaller than $q$.
Define $f : T_{m,n,q} \to W$
by
$$
f( \prod_{i,j} x_{i,j}^{p_{i,j}})
= (\tilde c_{1}, \ldots \tilde c_{n},
r_{2}, r_{3}, \ldots, r_{m}),
$$
where 
$\tilde c_{j} = (q-1) - \sum_{i} p_{i,j}$,
and $r_{i} = \sum_{j} p_{i,j}$.
First of all,
by assumption all $\tilde c_j$ are strictly smaller than $q$,
and
\begin{align*}
\sum_{i = 1}^{n} \tilde c_i + \sum_{j = 2}^{m} r_{j}
&= n(q-1) - \displaystyle\sum_{i,j} p_{i,j} + \sum_{i\neq 1,j} p_{i,j}  \\
&= n(q-1) - \sum_{j} p_{1,j} \\
& \leq n(q-1).
\end{align*}
Thus the image of $f$ is in $W$.
An $(n+m-1)$-tuple
$(\tilde c_{1}, \ldots \tilde c_{n},
r_{2}, r_{3}, \ldots, r_m) \in W$
uniquely identifies non-negative integers
$c_1, \ldots, c_n, r_2, \ldots, r_m$.
Furthermore,
$$
r_1 = \sum_{j=1}^n c_j - \sum_{i=2}^m r_i
= \sum_{j=1}^n (q-1-\tilde c_j) - \sum_{i=2}^m r_i
= n(q-1) - (\sum_{j=1}^n \tilde c_j + - \sum_{i=2}^m r_i)
\ge 0,
$$
so that
$(\tilde c_{1}, \ldots \tilde c_{n},
r_{2}, r_{3}, \ldots, r_m)$
uniquely identifies non-negative integers
$c_1, \ldots, c_n, r_1, \ldots, r_m$,
whence
by Lemma~\ref{lemstairmargs} it uniquely identifies
the staicase monomial $\prod_{i,j} x_{i,j}^{p_{i,j}}$.
Thus $f$ is injective and surjective,
so that by Lemma~\ref{lemWU},
\begin{align*}
N_{q}(m,n; \ol{\infty};\ol{q-1})
&= |W| \cr
&=
\sum_{i=0}^{n-1} (-1)^i \binom {n} {i} \binom {n(q-1) - iq + m + n -1} {m+n-1}
\cr
&=
\sum_{i=0}^{n-1} (-1)^i \binom {n} {n-i} \binom {(n-i)q + m -1} {m+n-1}
\cr
&=
\sum_{i=1}^{n} (-1)^{n-i} \binom {n} {i} \binom {iq +m-1} {m+n-1}.
\hfill\qed
\end{align*}

\begin{thm}\label{thm:mq}
For all positive integers $m, n$,
$$
M_{q}(m,n;\overline{q-1};\overline{q-1})
= \sum_{i=1}^n \sum_{j=0}^m (-1)^{m-j + i-1}
\binom n i \binom m j \binom{jq-iq + n - 1}{m+n-1}.
$$
\end{thm}

\proof
Let $S_{n,m,q}$ be the set of all staircase monomials
$\prod_{i,j} x_{i,j}^{p_{i,j}}$
such that either for all $i = 1, \ldots, m$,
$\sum_j p_{i,j} < q$
or for all $j \in \{1,\ldots,n\}$,
$\sum_{i} p_{i,j} <q$.
By Definition~\ref{defNq},
$N_{q}(m,n;\overline{\I};\overline{\I}) = |S_{n,m,q}|$.

Set $S_{m,n,q,k} :=
\big\lbrace
\prod_{i,j}x_{i,j}^{p_{i,j}} \in S_{m,n,q} :
k \leq |\{j:\sum_{i} p_{i,j} \geq q\}| \big\rbrace$.
Then $M_{q}(m,n;\overline{q-1}; \overline{q-1}) = |S_{m,n,q,1}|$.

For any $\prod_{i,j}x_{i,j}^{p_{i,j}} \in S_{m,n,q}$
let $c_j$ be the $j$th column sum $\sum_i p_{i,j}$
and $r_i$ the $i$th row sum $\sum_j p_{i,j}$.
Let $P_l(n)$ be the set of all $l$-element subsets of $\{1,\ldots, n\}$.
Let $W$ be the number of $(m+n-1)$-tuples
of non-negative integers that sum up to at most $m(q-1) - lq$
and for which the first $m$ entries are strictly smaller than $q$.

In the following we assume that $k > 0$.
If $\prod_{i,j}x_{i,j}^{p_{i,j}} \in S_{m,n,q,k} \setminus S_{m,n,q,k+1}$,
then by the definition of $S_{m,n,q}$,
$r_i \le q-1$ for all $i$.
Let $K \in P_k(n)$
such that $c_j \ge q$ if and only if $j \in K$.
For any $l$-element subset $L$ of $K$,
we associate
$(L, \prod_{i,j}x_{i,j}^{p_{i,j}})$
with $(L,q-1-r_1, \ldots, q-1-r_m, c'_1, \ldots, \widehat{c'_s}, \ldots, c'_n)$,
where $s$ is the smallest element in $L$,
and $c'_j = c_j-q$ if $j \in L$ and $c'_j = c_j$ otherwise.
Recall Lemma~\ref{lemWU}:
this associated element is in
$P_l(n) \times W$
because
all $r_i$ are at most $q-1$
and
\begin{align*}
\sum_{i=1}^m (q-1-r_i) + \sum_{j \not = s} c'_j
&= m(q-1) - \sum_{i=1}^m r_i + \sum_{j \not = s} c_j - (l-1)q \cr
&= m(q-1) - c_s - (l-1)q \cr
&\le m(q-1) - q - (l-1)q \cr
&= m(q-1) - lq.
\end{align*}

Conversely,
we can reverse this.
Namely,
let $(L,v) \in P_l(n) \times W$.
By definition the first $m$ entries in $v$ are non-negative integers
that are at most $q-1$,
and we can write them in the form $q-1-r_i$ for some
$r_i \in \{0, \ldots, q-1\}$.
We write the last $n-1$ entries of $v$
as $(c'_1, \ldots, c'_{s-1}, c'_{s+1}, \ldots, c'_n)$,
where $s$ is the smallest number in $L$.
For $j \in \{1, \ldots, n\} \setminus L$,
let $c_j = c'_j$,
for $j \in L \setminus \{s\}$,
set $c_j = c'_j + q$,
and finally set $c_s = (\sum_i r_i) - (\sum_{j \not = s} c_j)$.
Note that
$$
\sum_i r_i 
= m(q-1) - \sum_i (q-1 - r_i)
= m(q-1) -lq - \sum_i (q-1 - r_i) - \sum_{j \not =s} c'_j + lq + \sum_{j \not =s} c'_j
\ge lq + \sum_{j \not =s} c'_j
= q  + \sum_{j \not =s} c_j.
$$
Thus $c_s \ge q$ and so $c_i \ge q$ for all $i \in L$.
Furthermore,
by Lemma~\ref{lemstairmargs},
these non-negative numbers $r_1, \ldots, r_m, c_1, \ldots, c_n$
uniquely determine a staircase monomial
$\prod_{i,j}x_{i,j}^{p_{i,j}}$ uniquely.
Thus in any case,
$(L,v)$ yields $(L,\prod_{i,j}x_{i,j}^{p_{i,j}})$ uniquely.

We have proved the following:
for every staircase monomial
$\prod_{i,j}x_{i,j}^{p_{i,j}} \in S_{m,n,q,l}$
there exists a unique $k$ such that
$\prod_{i,j}x_{i,j}^{p_{i,j}} \in S_{m,n,q,k} \setminus S_{m,n,q,k+1}$.
There exist $\binom k l$ sets $L \in P_l(n)$
such that $c_j \ge q$ for all $j \in L$.
To $(L, \prod_{i,j}x_{i,j}^{p_{i,j}})$ we uniquely associate an element of
$P_l(n) \times W$.

Thus by Lemma~\ref{lemWU},
for all $l \ge 1$,
\begin{align*}
\sum_{k=l}^n \binom k l |S_{m,n,q,k} \setminus S_{m,n,q,k+1}|
&= |P_l(n) \times W| \cr
&= |P_l(n)| \cdot |W| \cr
&= \binom n l \sum_{j=0}^m (-1)^m
\binom m j \binom{m(q-1)-lq -jq + m + n - 1}{m+n-1} \cr
&= \binom n l \sum_{j=0}^m (-1)^m
\binom m j \binom{(m-j)q-lq + n - 1}{m+n-1} \cr
&= \binom n l \sum_{j=0}^m (-1)^{m-j}
\binom m j \binom{jq-lq + n - 1}{m+n-1}.
\end{align*}
This is a system of $n$ linear equations with matrix form:
$$
\begin{bmatrix}
\binom 1 1 & \binom 2 1 & \binom 3 1 & \cdots & \binom n 1 \cr
\binom 1 2 & \binom 2 2 & \binom 3 2 & \cdots & \binom n 2 \cr
\binom 1 3 & \binom 2 3 & \binom 3 3 & \cdots & \binom n 3 \cr
\vdots & \vdots & \vdots & \ddots & \vdots \cr
\binom 1 n & \binom 2 n & \binom 3 n & \cdots & \binom n n \cr
\end{bmatrix}
\begin{bmatrix}
\vphantom{\binom 0 0} |S_{m,n,q,1} \setminus S_{m,n,q,2}| \cr
\vphantom{\binom 0 0} |S_{m,n,q,2} \setminus S_{m,n,q,3}| \cr
\vphantom{\binom 0 0} |S_{m,n,q,3} \setminus S_{m,n,q,4}| \cr
\vdots \cr
\vphantom{\binom 0 0} |S_{m,n,q,n} \setminus S_{m,n,q,n+1}| \cr
\end{bmatrix}
=
\begin{bmatrix}
\binom n 1 \sum_{j=0}^m (-1)^{m-j} \binom m j \binom {jq - 1q + n - 1}{m+n-1} \cr
\binom n 2 \sum_{j=0}^m (-1)^{m-j} \binom m j \binom {jq - 2q + n - 1}{m+n-1} \cr
\binom n 3 \sum_{j=0}^m (-1)^{m-j} \binom m j \binom {jq - 3q + n - 1}{m+n-1} \cr
\vdots \cr
\binom n n \sum_{j=0}^m (-1)^{m-j} \binom m j \binom {jq - nq + n - 1}{m+n-1} \cr
\end{bmatrix}.
$$
But the inverse of this matrix $[\binom{j}{i}]_{i,j}$
is $[(-1)^{i+j}\binom{j}{i}]_{i,j}$,
so that by Cramer's rule,
\begin{align*}
M_{q}(m,n;\overline{q-1}; \overline{q-1})
&= |S_{m,n,q,1}| \cr
&= \sum_{k=1}^n |S_{m,n,q,k} \setminus S_{m,n,q,k+1}| \cr
&= \sum_{k=1}^n 
\sum_{i=1}^n
(-1)^{k+i} \binom i k
\binom n i
\sum_{j=0}^m (-1)^{m-j} \binom m j \binom{jq-iq + n - 1}{m+n-1} \cr
&= \sum_{i=1}^n 
(-1)^i \sum_{k=1}^n
(-1)^k \binom i k
\binom n i
\sum_{j=0}^m (-1)^{m-j} \binom m j \binom{jq-iq + n - 1}{m+n-1} \cr
&= \sum_{i=1}^n 
(-1)^{i-1}
\binom n i
\sum_{j=0}^m (-1)^{m-j} \binom m j \binom{jq-iq + n - 1}{m+n-1} \cr
&= \sum_{i=1}^n \sum_{j=0}^m (-1)^{m-j + i-1}
\binom n i \binom m j \binom{jq-iq + n - 1}{m+n-1} \cr
&= \sum_{i=1}^n \sum_{j=1}^m (-1)^{m-j + i-1}
\binom n i \binom m j \binom{jq-iq + n - 1}{m+n-1}. & \qed
\end{align*}

The main theorem on the generalized Hilbert-Kunz function now follows:

\begin{thm} \label{thmNinfinf}
For all positive $m, n$,
$\mbox{HK}_{R,X}(q)
= \length \displaystyle\left( \frac {k[X]} {I_{2}(X) + (x_{1,1}^q, \ldots, x_{m,n}^q) }\right)$
equals
\begin{align*}
N_{q}(m,n; \ol{\infty}; \ol{\infty})
&= N_{q}(m,n; \ol{\infty}; \ol{q-1}) + M_{q}(m,n;\ol{q-1};\ol{q-1}) \cr
&= \sum_{i=1}^{n} (-1)^{n-i} \binom {n} {i} \binom {iq +m-1} {m+n-1}
+
\sum_{i=1}^n \sum_{j=1}^m (-1)^{m-j + i-1}
\binom n i \binom m j \binom{jq-iq + n - 1}{m+n-1}.
\end{align*}
\end{thm}

\begin{proof}
By Definition~\ref{defNq},
$N_{q}(m,n; \ol{\infty}; \ol{\infty})$
counts all the staircase monomials
$\prod_{i,j} x_{i,j}^{p_{i,j}}$
with the property that either
for all $j$, $\sum_i p_{i,j} < q$,
or for all $i$, $\sum_j p_{i,j} < q$.

The number $N_{q}(m,n; \ol{\infty}; \ol{q-1})$
counts those monomials in the previous paragraph
for which for all $j$, $\sum_i p_{i,j} < q$,
and $M_{q}(m,n;\ol{q-1};\ol{q-1})$ counts those monomials
for which for some $j$, $\sum_i p_{i,j} \ge q$.
Thus
$N_{q}(m,n; \ol{\infty}; \ol{\infty})
= N_{q}(m,n; \ol{\infty}; \ol{q-1}) + M_{q}(m,n;\ol{q-1};\ol{q-1})$,
and by Theorems~\ref{thm:nq} and \ref{thm:mq} this equals to
the claimed sums of binomial coefficients.
\end{proof}

In particular,
comparison with Theorem 4.4 in \cite{MS} when $m = 2$
gives:

\begin{cor}\label{corcompid}
The number of $(n+1)$-tuples of non-negative integers
that sum up to at most $n(q-1)$ and for which the first $n$ entries
are strictly smaller than $q$ equals
$$
\sum_{i=1}^{n} (-1)^{n-i} \binom {n} {i} \binom {iq +1} {n+1}
= {nq^{n+1} - (n-2)q^n \over 2}.
$$
\end{cor}

\begin{proof}
According to \cite[Theorem 4.4]{MS},
$N_{q}(2,n; \ol{\infty}; \ol{\infty})
= {nq^{n+1} - (n-2)q^n \over 2} + n \binom{q + n - 1}{n+1}$,
and by Theorem~\ref{thmNinfinf},
\begin{align*}
N_{q}(2,n; \ol{\infty}; \ol{\infty})
&= \sum_{i=1}^{n} (-1)^{n-i} \binom {n} {i} \binom {iq +1} {n+1}
\cr&\hskip2em
+ 2 \sum_{i=1}^n (-1)^i \binom n i \binom{q-iq + n - 1}{n+1}
+ \sum_{i=1}^n (-1)^{i-1} \binom n i \binom{2q-iq + n - 1}{n+1} \cr
&= \sum_{i=1}^{n} (-1)^{n-i} \binom {n} {i} \binom {iq +1} {n+1}
+ \binom n 1 \binom{2q-1q + n - 1}{n+1} \cr
&= \sum_{i=1}^{n} (-1)^{n-i} \binom {n} {i} \binom {iq +1} {n+1}
+ n \binom{q + n - 1}{n+1}.
\end{align*}
Thus
$\sum_{i=1}^{n} (-1)^{n-i} \binom {n} {i} \binom {iq +1} {n+1}
= {nq^{n+1} - (n-2)q^n \over 2}$.
By Theorem~\ref{thm:nq},
this number is
the number of $(n+1)$-tuples of non-negative integers
that sum up to at most $n(q-1)$ and for which the first $n$ entries
are strictly smaller than $q$.
\end{proof}

We remark here that we know of no other proof of the equality
in the last corollary.
Natural first attempts would be induction
and Gosper's algorithm,
and neither of these is successful,
as for one thing, the summands depend not only on the summing index $i$
but also on $n$.
The challenge remains to establish a closed-form expression
for $N_q(m,n; \overline \infty; \overline \infty)$ and
$N_q(m,n; \ol \infty; \ol {q-1})$ for higher $m$.

\end{document}